\documentclass{amsart}
\usepackage{amsmath}
\usepackage{amsthm}
\usepackage{amsfonts}
\usepackage{amssymb}
\usepackage{bbm}
\usepackage{graphicx}
\usepackage{tikz,subfig}
\usepackage{natbib}
\usepackage{enumitem}
\usepackage{tikz-cd}
\usepackage{algorithmic,algorithm}
\usepackage[text={440pt,575pt},headheight=9pt,centering]{geometry}
\usepackage[textwidth=2.8cm]{todonotes} 
\usepackage{mathtools}
\usepackage{bm}
\usepackage{comment}
\makeatletter

\usepackage{hyperref}
\hypersetup{
	colorlinks=true,
	linkcolor=blue,
	filecolor=magenta,
	urlcolor=cyan,
    citecolor=blue,
}
\usepackage[capitalize]{cleveref}

\allowdisplaybreaks

\newcommand{\x}{\mathbf{x}}
\newcommand{\y}{\mathbf{y}}

\newcommand{\p}{\mathbf{p}}
\newcommand{\1}{\mathbf{1}}
\newcommand{\0}{\mathbf{0}}
\newcommand{\X}{\mathbf{X}}
\newcommand{\Y}{\mathbf{Y}}

\newcommand{\RR}{\mathbb{R}}

\newcommand{\PP}{\mathbb{P}}

\newcommand{\CM}{\operatorname{CM}}
\newcommand{\diag}{\text{diag}}

\newcommand{\EMTPtwo}{$ \text{EMTP}_2 $ }
\newcommand{\EMTPtwok}{$ \text{EMTP}_2 $, }

\newcommand{\indep}{\perp \!\!\! \perp}

\newcommand{\ones}{\mathbf{1}}

\newtheorem{thm}{Theorem}[section]
\newtheorem{prop}[thm]{Proposition}
\newtheorem{cor}[thm]{Corollary}
\newtheorem{lemma}[thm]{Lemma}

\theoremstyle{definition}
\newtheorem{defi}{Definition}
\newtheorem{ex}{Example}

\newtheorem{remark}[thm]{Remark}

\newcommand{\spa}{\rm{span}}

\newcommand{\HR}{\rm{HR}}

\begin{document}
	\title[Extremal conditional independence for H\"usler--Reiss distributions via modularity]{Extremal conditional independence for H\"usler--Reiss distributions via modular functions}
     \author{Karel Devriendt}
    \address{Mathematical Institute, University of Oxford, Oxford, United Kingdom}
	\author{Ignacio Echave-Sustaeta Rodr\'iguez}
    \address{Department of Mathematics and Computer Science, Eindhoven University of Technology, Eindhoven, The Netherlands}
    \author{Frank R\"ottger}
    \address{Department of Applied Mathematics, University of Twente, Enschede, The Netherlands}

\begin{abstract}
We study extremal conditional independence for H\"{u}sler--Reiss distributions, which is a parametric subclass of multivariate Pareto distributions.
As the main contribution, we introduce two set functions, i.e.~functions which assign a value to the distribution and each of its marginals, and show that extremal conditional independence statements can be characterized by modularity relations for these functions.
For the first function, we make use of the close connection between H\"{u}sler--Reiss and Gaussian models to introduce a multiinformation-inspired measure $m^{\HR}$ for H\"{u}sler--Reiss distributions. 
For the second function, we consider an invariant $\sigma^2$ that is naturally associated to the H\"{u}sler--Reiss parameterization and establish the second modularity criterion under additional positivity constraints.
Together, these results provide new tools for describing extremal dependence structures in high-dimensional extreme value statistics.
In addition, we study the geometry of a bounded subset of H\"usler--Reiss parameters and its relation with the Gaussian elliptope.
\end{abstract}
\maketitle
\section{Introduction}

When several components of a system become large at the same time, decisions depend on how these extremes occur jointly rather than on each margin in isolation. Examples include simultaneous river floods, joint price spikes, or infrastructure failures. In such problems one wants models that describe which subsets can become large together and which parts of the system effectively decouple once others are known.
A popular model for such questions are generalized multivariate Pareto distributions, which can be obtained as the limit of threshold exceedances \citep{roo2006}.
Here, we consider a multivariate observation as extreme if at least one entry exceeds a high threshold $u$, where the limit is taken as $u\to\infty$.
While this approach allows to model multivariate extremes of large systems, its support is by construction not a product space.
Thus, standard notions of conditional independence and graphical models are not applicable.
A suitable alternative called \emph{extremal conditional independence} was introduced by \citet{engelkehitz}, which sparked a new research direction for extremal dependence modeling. We refer to the review article of \citet{engelke2024graphicalmodelsmultivariateextremes} for further details and references.

In this paper, we study extremal conditional independence for a parametric subclass of multivariate Pareto distributions.
This subclass is the family of H\"usler--Reiss distributions, which can be parametrized by a \emph{variogram} matrix~$\Gamma$. 
This family of distributions can be considered as an extremal analogue of the multivariate Gaussians, due to many remarkable properties.
For example, if the threshold exceedances in a system converge to a H\"usler--Reiss distribution, the limit of threshold exceedances of any subsystem will also be H\"usler--Reiss, parameterized by the corresponding submatrix of $\Gamma$.
Furthermore, the H\"usler--Reiss distribution allows an alternative parameterization by a \emph{precision matrix}~$\Theta$, which can be interpreted as the Laplacian matrix of a signed graph~\citep{devriendt2022a}.
As H\"usler--Reiss distributions allow a parametric encoding of extremal conditional independence, they give rise to extremal graphical models, where zeros in~$\Theta$ encode the pairwise extremal Markov property \citep{HES2022}. 

As the main contribution of this paper we introduce two criteria that encode arbitrary extremal conditional independence statements for H\"usler--Reiss distributions.
First, we introduce a general criterion based on Gaussian multiinformation.
Here, we exploit that H\"usler--Reiss distributions can be characterized by multivariate Gaussians where both the covariance and mean depend on the variogram matrix $\Gamma$.
For a $d$-variate distribution, this gives rise to a set function $m^{\HR}$ that assigns to each subset $I\subset\{1,\dots,d\}$ of indices the value
$$
m^{\HR}(I) \,:=\, -\frac{1}{2}\log\det\begin{pmatrix}
    -\Gamma_{I,I}/2 & \ones\\
    -\ones^\top &0\\
\end{pmatrix}.
$$
The main result is that this set function satisfies a linear modularity relation of the form $m^{\HR}(A B C)+m^{\HR}(C)=m^{\HR}(A B)+m^{\HR}(B C)$ if and only if there is an extremal conditional independence relation between the nonempty disjoint subsets $A,B,C$ of variables; we refer to Theorem \ref{THM:MULTIINFO} for details and notation.
We highlight that this provides the first parametric encoding of extremal conditional independence for multivariate conditioned sets, since all previous results were restricted to conditioned sets that are singletons \citep{engelke2024graphicalmodelsmultivariateextremes,EGR2025}.

The second set function is based on an invariant that naturally shows up in the H\"{u}sler--Reiss parameterization and its exponent measure density. For a $d$-variate distribution with variogram $\Gamma$, this gives rise to a set function $\sigma^2$ that assigns to each subset $I\subset\{1,\dots,d\}$ of indices the value
$$
\sigma^2(I) := \frac{1}{2}\big(\1^T\Gamma_{I,I}^{-1}\1\big)^{-1}.
$$
Under the condition that the associated precision matrix $\Theta$ has nonpositive off-diagonal entries and that $\Gamma^{-1}$ has positive row sums, we show that $\sigma^2$ satisfies modularity relations if and only if an associated conditional independence relation holds; this is Theorem \ref{THM:SIGMASQ}. 
We illustrate both results in detail for H\"{u}sler--Reiss graphical models for the $4$-cycle graph. We note that the sign restriction on $\Theta$ is equivalent to extremal multivariate total positivity of order 2 ($\text{EMTP}_2$) for H\"usler--Reiss distributions, which is an important notion of positive dependence \citep{REZ2023}. If $\Gamma$ is interpreted in terms of resistance geometry, the set function $\sigma^2$ is called the resistance radius, while the condition that $\Gamma^{-1}$ has positive rowsums imposes positive resistance curvature.
We refer to \citet{devriendt2022a} for further reading on resistance geometry.

Finally, we investigate the geometry of variogram matrices and introduce the H\"usler--Reiss elliptope as a natural choice for a bounded geometric object defined by variogram matrices. We discuss the relationship of the H\"usler--Reiss elliptope with the classical Gaussian elliptope, and visualize our findings in three dimensions.

\subsection*{Outline} The paper is structured in five sections.
In Section~\ref{sec:preliminaries} we recall the necessary preliminaries for threshold exceedances, H\"usler--Reiss distributions and extremal conditional independence.
Next, we derive our first main result, which shows the equivalence of arbitrary extremal conditional independence and the modularity of $m^{\HR}$, in Section~\ref{sec:multiinfo}.
The second main result on the modularity of $\sigma^2$ and its relation to extremal conditional independence for H\"usler--Reiss distributions is discussed in Section~\ref{SUBSEC:MODULARITYSIGMA}.
Finally, in Section~\ref{sec:geometry}, we discuss the geometry of a bounded subset of valid variogram matrices and its connection with the Gaussian elliptope.

\subsection*{Notation}
Unions of sets $A\cup B$ are abbreviated as $AB$. We write $[d]:=\{1,\dots,d\}$ and the complement of a set $[d]\setminus I $ is written as $ \backslash I$. 
Submatrices of a matrix are written with comma between row/column index, e.g. $\Gamma_{I,I}$ or $\Gamma_{Ai,Bj}$. A general subset of $[d]$ is $I\subseteq[d]$. A general triple of subsets is $A,B,C\subseteq [d]$.

\section{Preliminaries}\label{sec:preliminaries}

\subsection{Multivariate extremes and extremal conditional independence}\label{SUBSEC:EXTREMES}

Let $\X=(X_1,\dots,X_d)$ be a $d$-variate random vector, where each margin $X_i$ follows a continuous probability distribution. In extremal dependence modeling we are interested in modeling simultaneous occurrences of rare events.
If the limit exists, the extremal correlation coefficient
\[\chi_{ij} := \lim_{q\to 0}\mathbb{P}( F_i(X_i)>1-q\mid F_j(X_j)>1-q ),\]
where $F_i$ is the cumulative distribution function of $X_i$,
is a bivariate summary statistic for extremal dependence.
If $\chi_{ij}>0$, we call $X_i$ and $X_j$ asymptotically dependent, while for $\chi_{ij}=0$, one speaks of asymptotic independence.
In order to focus on the extremal dependence structure we define a standardized vector $\X^*$ via $X_i^* =-\log(1-F_i(X_i))$, such that all $X_i^*$ follow a standard exponential distribution.

In order to model the extremal behaviour of $\X^*$, we use the approach of multivariate threshold exceedances. Here, one considers any observation of $\X^*$ as extreme if at least one of its components exceeds a high threshold $u\in \RR$.
The limiting distribution of this approach, if it exists \citep{roo2006}, is a (generalized) multivariate Pareto distribution 
\begin{align}
    \PP(\Y\le \y) = \lim_{u\to \infty} \PP\big(\X^* - u\ones \le \y \mid \max_{i\in [d]} X_i^* \ge u\big), \label{eq:MPD}
\end{align}
where $\ones=(1,\ldots,1)^\top$ is the all-ones vector. Under the assumption of asymptotic dependence, the distribution of $\Y$ is supported on the $L$-shaped space 
$$
\mathcal{L}:=\big\{\y\in\RR^d:\max_{i \in [d]} y_i > 0\big\}
$$ 
and has a probability density with respect to the Lebesgue measure on $\RR^d$.
The density of $\Y$ is proportional to the density $\lambda(\y)$ of the exponent measure corresponding to $\Y$, restricted to $\mathcal{L}$, see e.g.~\citet{engelke2024graphicalmodelsmultivariateextremes} for details.
When $\Y$ is restricted to a positive canonical halfspace $\mathcal{L}^k:=\{\y\in\mathcal{L}:y_k>0\}$ for some $k\in [d]$, the resulting vector $\Y^k:=\Y|\{Y_k>0\}$ has probability density $\lambda(\y)$ on its support $\mathcal{L}^k$. Thus, for this paper, we can treat $\lambda(\y)$ as a function whose restriction to any positive canonical halfspace gives the probability density of $\Y^k$.

For any subset $I\subseteq [d]$, the limiting distribution \eqref{eq:MPD} of a marginal vector $\X_I^*$ is again a multivariate Pareto vector $\Y_{(I)}$ whose exponent measure density is given by $\lambda_I(\y_I)=\int_{\RR^{d-|I|}}\lambda(\y)d\y_{[d]\setminus I}$; we note that $\Y_{(I)}$ typically follows a different distribution than the marginal vector $(Y_i)_{i\in I}$ of $\Y$. This notion of marginalization gives rise to conditional densities $\lambda(\y_A|\y_C):=\frac{\lambda_{AC}(\y_{AC})}{\lambda_{C}(\y_{C})}$, where we recall that $AC$ abbreviates the union $A\cup C\subseteq [d]$. Because $\Y^k$ has density $\lambda(\y)$ on its support $\mathcal{L}^k$ and because the latter is a product space, the conditional vector $\Y^k_A\vert\{\Y^k_C=\y_C\}$ has density $\lambda(\y_A|\y_C)$, for any nonempty disjoint sets $A,C\subset [d]$ with $k\in C$.
\\~\\

Since the support $\mathcal{L}$ of $\Y$ is not a product space, the standard notion of conditional independence is not meaningful for multivariate Pareto distributions.
\citet{engelkehitz} proposed an alternative based on factorization of the exponent measure density, which is equivalent to standard conditional independence for the collection of vectors $\{\Y^k\}_{k\in[d]}$.
\begin{defi}[{\cite[Definition~5, Proposition~1]{engelkehitz}}]\label{definition: extremal CI}
	Let $\Y$ be a multivariate Pareto vector with exponential margins, and denote its exponent measure density by $\lambda$. Let $A,B,C$ be a partition of the index set. We say that $\Y_A$ is \textit{extremal conditionally independent}
	of $\Y_B$ given $\Y_C$, denoted by $\Y_A \perp_e \Y_B \mid \Y_{C}$, if and only if 
	\begin{equation*}
		\lambda(\y)\lambda_C(\y_C) \,=\, \lambda_{AC}(\y_{AC})\lambda_{BC}(\y_{BC}) \quad\text{~for all $\mathbf{y}\in\mathcal{L}$.}
	\end{equation*}

\end{defi}
The extremal conditional independence condition can also be stated in terms of conditional densities as $\lambda(\y_{AB}|\y_C)=\lambda(\y_A|\y_C)\lambda(\y_B|\y_C)$. The notion of extremal conditional independence in Definition~\ref{definition: extremal CI} allows for flexible modeling of dependence structures in multivariate extremes. In particular, it gives rise to directed and undirected notions of extremal graphical models. If $G=(V,E)$ is an undirected graph, we say that a multivariate Pareto vector $\Y$ satisfies the \emph{undirected extremal global Markov property} with respect to $G$, when for any disjoint $A,B\subset V$ with separating set $C$, we have $\Y_{A}\perp_e\Y_B\mid \Y_C$; here, the margins $[d]$ are identified with the vertex set $V$. A vector $\Y$ satisfies the \textit{undirected extremal pairwise Markov property} with respect to $G$, when for any non-edge $ij\notin E$, we have $Y_i \perp_e Y_j \mid \Y_{V\backslash\{i,j\}}$.
For further details on extremal graphical models we refer to \citet{engelkehitz,REZ2023,RCG2023,HES2022,EGR2025}.

\subsection{H\"usler--Reiss distributions}\label{SUBSEC:HR}
In this article we focus on a parametric subclass of multivariate Pareto distributions, which can be considered as an extremal analogue to multivariate Gaussian distributions. Similar to how a Gaussian distribution is parametrized by its covariance matrix, a \textit{H\"usler--Reiss} multivariate Pareto distribution is parametrized by its \emph{variogram} $\Gamma$. A variogram is a conditionally negative definite matrix which, for $d$-variate distributions, is a matrix in the following set
$$
\mathcal{D}^d \,:=\, \big\{\Gamma \in \mathbb{R}^{d\times d}\,:\, \Gamma=\Gamma^T,~ \diag(\Gamma)=0\text{~and~} \mathbf{x}^T\Gamma\mathbf{x}<0~\forall \0\neq\mathbf{x}\perp\mathbf{1}\big\}.
$$
The exponent measure density of a H\"{u}sler--Reiss distribution with variogram $\Gamma$ is given by:
\begin{equation}\label{eq: HR exponent measure density}
\lambda(\y) \,:=\, \sqrt{\frac{-(2\pi)^{1-d}}{\det\CM( \Gamma )}}\exp\left(-\frac{1}{2}\begin{pmatrix}
    \y & 1 \end{pmatrix}\CM(\Gamma)^{-1}\begin{pmatrix}
    \y\\
    1
\end{pmatrix}\right) \quad\text{~where~}\quad \CM(\Gamma):=\begin{pmatrix}-\Gamma/2&\1\\\1^T&0\end{pmatrix}.
\end{equation}
Here, the matrix $\CM(\Gamma)$ is called the \textit{Cayley--Menger matrix}\footnote{The name ``Cayley--Menger matrix" is adopted from the context of distance geometry, where the matrix $\CM(-2D)$ associated to a Euclidean distance matrix $D$ plays an important role.} of $\Gamma$. An alternative way to parametrize H\"{u}sler--Reiss distributions is via its \emph{precision matrix} $\Theta$. This is a positive semidefinite matrix with $\ker(\Theta)=\spa(\1)$ and can be obtained from the variogram via the \emph{Fiedler--Bapat identity} \citep{devriendt2022a}
\begin{equation}\label{eq: Fiedler--Bapat identity}
\begin{pmatrix}
    \Theta & \p\\\p^T&\sigma^2
\end{pmatrix}\,=\,\begin{pmatrix}
-\Gamma/2 & \1\\\1^T&0
\end{pmatrix}^{-1},
\end{equation}
where the vector $\p$ and scalar $\sigma^2$ are fully determined by the precision matrix or variogram; see also Section~\ref{SUBSEC:MODULARITYSIGMA}. In terms of the precision matrix, the H\"{u}sler--Reiss exponent measure density is written as
\begin{equation}\label{eq: HR density precision matrix}
\lambda(\mathbf{y}) \,=\, \sqrt{(2\pi)^{1-d}\det( \Theta_{\backslash k,\backslash k})}\exp\left(-\frac{1}{2}\begin{pmatrix}\y&1\end{pmatrix}\begin{pmatrix}
\Theta&\p\\\p^T&\sigma^2
\end{pmatrix}\begin{pmatrix}
\y\\1 \end{pmatrix}\right),
\end{equation}
where $k$ can be any element of $[d]$ and we recall the abbreviation $\backslash k:=[d]\backslash k$.

H\"usler--Reiss distributions have the special property that their marginals $\Y_{(I)}$, i.e., the limiting distribution of $\X^*_I$ in \eqref{eq:MPD} for some $I\subseteq[d]$, again follow a H\"usler--Reiss distribution and, moreover, that the corresponding variogram is equal to the principal $I$-submatrix $\Gamma_{I,I}$ of $\Gamma$. For the corresponding precision matrix, we use the notation
$$
\begin{pmatrix}
\Theta(I)&\p(I) \\\p(I)^T&\sigma(I)^2
\end{pmatrix}\,:=\,\begin{pmatrix}
-\Gamma_{I,I}/2&\1\\\1^T&0
\end{pmatrix}^{-1}.
$$
We note that with this notation, we have $\Theta([d])=\Theta, \p([d])=\p$ and $\sigma^2([d])=\sigma^2$. In Section~\ref{SUBSEC:MODULARITYSIGMA} we will say more about how these $I$-dependent quantities can be computed for general subsets $I\subseteq [d]$.
\\~\\
The specific definition of the H\"usler--Reiss exponent measure density allows one to derive a closed-form expression for conditional exponent measure densities. For any two nonempty disjoint subsets $A,C\subset [d]$, 
the conditional exponent measure density $\lambda(\y_A | \y_{C})$ is Gaussian with covariance matrix and mean:
$$
	\Sigma^* =-\frac{1}{2}\Gamma_{A,A} -\begin{pmatrix}
	-\frac{1}{2}\Gamma_{A,C}& \1\\
	\end{pmatrix}
	\CM(\Gamma_{C,C})^{-1}
	\begin{pmatrix}
	-\frac{1}{2}\Gamma_{C,A}\\
	\1^{\top}\\
	\end{pmatrix} 
    \;\text{~and~}\;
    \mu^*=\begin{pmatrix}
	-\frac{1}{2}\Gamma_{A,C}& \1\\
	\end{pmatrix}\CM(\Gamma_{C,C})^{-1}\begin{pmatrix}
	\y_{C}\\
	1\\
	\end{pmatrix}, 
$$
see \citet[Appendix C, Lemma~2]{EGR2025}. Note that the covariance matrix $\Sigma^*$ is a Schur complement of the Cayley--Menger matrix, resembling the covariance structure of conditional Gaussian distributions.

Due to this characterization of $\lambda(\y_A|\y_C)$ as a Gaussian density, the H\"usler--Reiss distribution allows for a parametric description of arbitrary extremal conditional independence statements; we recall the notation $AB=A\cup B$ and $\backslash A=[d]\setminus A$.
When used as indices, we display a singleton set $\{i\}$ as $i$. 

\begin{lemma}[{\citet[Appendix C, Lemma~3]{EGR2025}, \citet[Proposition~3.3]{RCG2023}}]\label{LEM:APPENDIXCEGR}
Let $\Y$ be a Hüsler--Reiss vector with variogram $\Gamma$ and precision matrix $\Theta$. Then for any nonempty disjoint subsets $\{i\},\{j\},C\subset[d]$, the following are equivalent: 
\begin{enumerate}
    \item[(i)] $Y_i\perp_e Y_j\mid \Y_C$,
    \item[(ii)] $\Theta(Cij)_{ij} = 0,$ 
    \item[(iii)] $\det \left( \CM(\Gamma_{Ci,Cj})\right) \;=\; 0,$
    \item[(iv)] $\det\bigl(\Theta_{\backslash Ci,\backslash Cj}\bigr) = 0.$
\end{enumerate}
\end{lemma}
In particular, for the case $Cij=[d]$ this characterization shows that for a H\"{u}sler--Reiss vector $\Y$, pairwise extremal conditional independence can be described by its precision matrix:
$$
Y_i \perp_e Y_j \mid \Y_{\backslash ij} \iff \Theta_{ij} = 0.
$$
Thus for a given graph $G=(V,E)$, the pairwise extremal Markov property for H\"{u}sler--Reiss graphical models is encoded by the condition that $ij\not\in E \Rightarrow \Theta_{ij}=0.$ This provides a parametric representation of the graphical model, see also \citet{HES2022}.

\begin{ex}[Four-cycle]\label{example: four-cycle}
Let $\Y$ be a H\"usler--Reiss graphical model with respect to the four-cycle $G$ as shown below, and let $\Theta$ be its precision matrix
$$
\raisebox{-2.4em}{\includegraphics[width=0.115\textwidth]{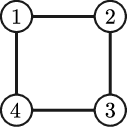}}
\quad\quad\quad\quad\quad\quad\quad\quad
\Theta=\begin{pmatrix}
        \Theta_{11}&\Theta_{12}&0&\Theta_{14}\\
        \Theta_{12}&\Theta_{22}&\Theta_{23}&0\\
        0&\Theta_{23}&\Theta_{33}&\Theta_{34}\\
        \Theta_{14}&0&\Theta_{34}&\Theta_{44}\\
    \end{pmatrix}.   
$$
The zeros $\Theta_{13}=0$ and $\Theta_{24}=0$ in the precision matrix encode the extremal conditional independencies $Y_1\perp_e Y_3 \mid \Y_{24} $ and $ Y_2\perp_e Y_4 \mid \Y_{13}$ prescribed by the non-edges $13,24\not\in E(G)$ of the graph. 
\end{ex}

To conclude this section, we mention a special subclass of H\"{u}sler--Reiss distributions. If the precision matrix $\Theta$ is an $M$-matrix, i.e., if all its off-diagonal entries are non-positive, then the corresponding H\"{u}sler--Reiss distribution is said to satisfy \textit{extremal multivariate total positivity of order 2}, abbreviated $\text{EMTP}_2$. This subclass was introduced in \citet{REZ2023} as an extremal analogue to strong positive dependence in the classical setting. The EMTP$_2$ property is preserved under marginalization, and it implies many additional properties. One reason for this is that an EMTP$_2$ precision matrix $\Theta$ can be interpreted as the Laplacian matrix of a positively weighted graph.

\section{Extremal conditional independence via Gaussian multiinformation}\label{sec:multiinfo}

We first recall the concept of multiinformation and some relevant results for this article; we refer to \cite{Studeny_2005} as our main reference.
\begin{defi}[Relative entropy]
Let $P$ and $Q$ be two probability measures with probability densities $p$ and $q$, defined over the same space $\mathcal{X}\subseteq \RR^d$ such that $P$ is absolutely continuous with respect to $Q$. Then, the \textit{relative entropy of $P$ with respect to $Q$} is defined as 
$$
H(P| Q) := \int_{\mathcal{X}}p(x)\log\left(\frac{p(x)}{q(x)}\right)dx.
$$
\end{defi} 
Let $\X$ be a $d$-variate random vector that follows a continuous probability measure $P$ with density $p$. Denote by $P_I$ and $p_I$ the probability measure and density of the marginal vector $\X_I$ for any $I\subseteq [d]$. Given a partition $A, B, C$ of $[d]$, the vector $\X_A$ is conditionally independent of $\X_B$ given $\X_C$ if and only if $p(\x_{AB}\vert \x_C) = p(\x_A\vert \x_C)p(\x_B\vert \x_C)$. We denote this by $\X_A \indep \X_B | \X_C$. 

\begin{defi}[Multiinformation]\label{def:multiinfo}
Let $P$ be a probability measure defined over a space $\mathcal{X}\subseteq \RR^d$. The \textit{multiinformation} $m_P$ of a subset $I\subseteq [d]$ is defined as:
$$
m_P(I) := H(P_{I}|\prod\limits_{i \in I}P_i).
$$
\end{defi}
We note that this concept was originally defined in \citet{Watanabe1960TotalCorrelation} under the name total correlation. The multiinformation function allows for an alternative description of conditional independence statements in terms of modularity equations, as illustrated by the following theorem:

\begin{thm}[{\citet[Corollary~2.2]{Studeny_2005}}]\label{THM:CIMULTI}
Let $\X$ be a random vector with probability measure $P$ and finite multiinformation. Then for any disjoint $A,B,C\subset [d]$ the multiinformation is supermodular
$$
m_P(ABC) + m_P(C)  \geq m_P(AC) + m_P(BC),
$$ 
with equality ($=$ modular) if and only if the subsets satisfy the conditional independence relation:
\begin{equation}\label{eq:multiinfo-modular}
m_P(ABC) + m_P(C) - m_P(AC) - m_P(BC) = 0 \iff \X_A \indep \X_B \mid \X_C. 
\end{equation}
\end{thm}
Note that for an empty conditioning set $C$ in \eqref{eq:multiinfo-modular}, the marginal independence statement $\X_A \indep \X_B$ is equivalent to
$$
m_P(AB) - m_P(A) - m_P(B) = 0 \quad\text{~for any disjoint $A,B\subset [d]$.}
$$

\begin{ex}[Gaussian distributions]
 Let $\X$ be a $d$-variate Gaussian distribution with correlation matrix $R$, the multiinformation function has the explicit expression
\begin{align}
    m_P(I) = -\frac{1}{2}\log(\det(R_{I,I})), \label{eq:GaussianMI}
\end{align}
for any $I\subseteq [d]$.
Thus, for any disjoint $A,B,C\subseteq [d]$ we obtain that $\X_A\indep \X_B \mid \X_C$ is equivalent to
\[-\frac{1}{2}\log(\det(R_{ABC,ABC})) -\frac{1}{2}\log(\det(R_{C,C})) +\frac{1}{2}\log(\det(R_{AC,AC})) +\frac{1}{2}\log(\det(R_{BC,BC})) = 0.\]
\end{ex}

\subsection{Extremal conditional independence for H\"usler--Reiss distributions}
In this section, we introduce a set function that encodes extremal conditional independence for H\"usler--Reiss distributions via modularity, similar to Theorem \ref{THM:MULTIINFO}. Due to the shape of the support $\mathcal{L}$ of a H\"usler--Reiss vector $\Y$, densities can never be fully factorized and the standard definition of multiinformation cannot be used. As an alternative, we propose a set function that encodes extremal conditional independence relations for H\"usler--Reiss distributions based on the Gaussianity of conditional exponent measure densities (see \Cref{SUBSEC:HR}). 
\begin{defi}
Let $\Y$ be a H\"usler--Reiss vector with variogram $\Gamma$. For $\emptyset\neq I\subseteq [d]$, we define $m^{\HR}(I)$ as
$$
m^{\HR}(I) \,:=\, -\frac{1}{2}\log\det\begin{pmatrix}
    -\Gamma_{I,I}/2 & \ones\\
    -\ones^\top &0\\
\end{pmatrix},
$$
and we let $m^{\HR}(\emptyset):=0$.
\end{defi}

Note that $m^{\HR}$ is not the actual multiinformation for the H\"{u}sler--Reiss distribution, as its support is not a product space. We start by showing a few basic properties of this function.
\begin{prop}\label{prop: properties of mHR}
The function $m^{\HR}$ satisfies
\begin{itemize}
    \item[(i)] $m^{\HR}(i)=0$ for all $i\in I$,
    \item[(ii)] $m^{\HR}(ij)=-\frac{1}{2}\log(\Gamma_{ij})$ for all $ij\subseteq [d]$,
    \item[(iii)] $m^{\HR}(I)-m^{\HR}(I\backslash i) = \frac{1}{2}\log(\Theta(I)_{ii})$ for all $i\in I\subseteq[d]$ of size $\vert I\vert>1$.
\end{itemize}
\end{prop}
\begin{proof}
Expressions (i) and (ii) follow by direct computation. Expression (iii) follows by writing out the difference on the lefthandside as
$$
m^{\HR}(I)-m^{\HR}(I\backslash i) = \frac{1}{2}\log\left(\frac{\det\CM(\Gamma_{I\backslash i,I\backslash i})}{\det\CM(\Gamma_{I,I})}\right) = \frac{1}{2}\log(\Theta(I)_{ii}),
$$
where the last step follows by the Fiedler--Bapat identity, and the cofactor formula for matrix inversion.
\end{proof}
\begin{remark}
Property (ii) in Proposition \ref{prop: properties of mHR} implies that the set function $m^{\HR}$ contains all information about the parameterization: all entries of the variogram can be retrieved by the transformation $\Gamma_{ij}=e^{-2m^{\HR}(ij)}$. As a result, one may think of the function $m^{\HR}$ on subsets of $[d]$ as yet another parameterization.
\end{remark}

We note that $m^{\HR}$ is not standardized to nonnegative values.
We can derive a standardization for any $k\in I$. Let 
\[
\Sigma^{(k)} := -\frac{1}{2}\Gamma_{I\setminus k,I\setminus k} -\begin{pmatrix}
	-\frac{1}{2}\Gamma_{I,k}& \1\\
	\end{pmatrix}
	\CM(\Gamma_{k,k})^{-1}
	\begin{pmatrix}
	-\frac{1}{2}\Gamma_{k,I}\\
	\1^{\top}\\
\end{pmatrix} 
\]
be the covariance matrix of the conditional H\"usler--Reiss exponent measure density $\lambda(\y_{I\setminus k}|y_k)$.
Then, by standard Schur complement arguments, $m^{\HR} (I)=-\frac{1}{2}\log\det (\Sigma^{(k)})$.
If we standardize $\Sigma^{(k)}$ to a correlation matrix, the resulting Gaussian multiinformation is equal to
$m^{\HR}(I)+\frac{1}{2}\sum_{i\in I\setminus k}\log \left(\Sigma^{(k)}_{ii}\right)$,
which is nonnegative. As this function depends on $k$ and our interest is in extremal conditional independence, we omitted this standardization in the definition of $m^{\HR}$. 

In addition to the alternative representation of $m^{\HR}$ based on the covariance matrix $\Sigma^{(k)}$, we collect further alternative representations in the following lemma.
\begin{lemma}
The following expressions are equivalent for the function $m^{\HR}$ and any nonempty $I\subseteq[d]$:
\begin{enumerate}
\item[(i)] $m^{\HR}(I) \,=\, -\frac{1}{2}\log\big(-\det\CM(\Gamma_{I,I})\big),$
\item[(ii)] $m^{\HR}(I)=\frac{1}{2}\log\det(\Theta(I)_{\backslash k,\backslash k})$ for any $k\in I$,
\item[(iii)] $m^{\HR}(I)=\frac{1}{2}\log\operatorname{Det}(\Theta(I))-\frac{1}{2}\log (|I|)$, where $\operatorname{Det}$ is the pseudo-determinant, i.e.~the product of all nonzero eigenvalues.
\item[(iv)] $
m^{\HR}(I) = -\log\int_{y_k>0} \exp\left(-\frac{1}{2}\begin{pmatrix}
    \y_I & 1 \end{pmatrix}\CM(\Gamma_{I,I})^{-1}\begin{pmatrix}
    \y_I\\
    1
\end{pmatrix}\right) d\y_I+\frac{|I|-1}{2}\log(2\pi)$ for any $k\in I$.
\item[(v)] $m^{\HR}(I) = \tfrac{1}{2}\log\big(\sum_T \prod_{ij\in E(T)}(-\Theta(I)_{ij})\big)$, where $T$ runs over all spanning trees of the graph corresponding to $\Theta(I)$.
    \end{enumerate}
\end{lemma}
\begin{proof}
    Expression {(i)} follows from the definition of $\CM(\Gamma)$ and basic identities for determinants. Expression (ii) follows by observing that $\Theta(I)_{\backslash k,\backslash k}$ is the inverse of the matrix $\Sigma^{(k)}$ defined above. For expression (iii), we use the identity $\operatorname{Det}(\Theta(I))=|I|\det (\Theta(I)_{\backslash k,\backslash k})$, as derived in \cite[Eq. (23)]{REZ2023}.
    Expression (iv) follows from the fact that the H\"usler--Reiss exponent measure density \eqref{eq: HR density precision matrix} is a probability density on each canonical halfspace $\mathcal{L}^k$, and thus
    $$
    \int_{y_k>0}\exp\left(-\tfrac{1}{2}\begin{pmatrix}\y_I&\1
        
    \end{pmatrix}^T\CM(\Gamma_{I,I})^{-1}\begin{pmatrix}
        \y_I\\\1
    \end{pmatrix}\right)d\y_I = \left((2\pi)^{1-|I|}\det( \Theta(I)_{\backslash k,\backslash k})\right)^{-\frac{1}{2}}.
    $$
    Introducing (ii) then leads to expression (iv). Finally, expression (v) follows from (ii) and the weighted matrix-tree theorem, see e.g.~\cite{REZ2023}.
\end{proof}

In the following theorem, we show that $m^{\HR}$ is supermodular and can encode arbitrary extremal conditional independence for H\"usler--Reiss distributions via modularity.
\begin{thm}\label{THM:MULTIINFO}
Let $\Y$ be a H\"usler--Reiss vector. Then for any nonempty disjoint $A,B,C\subset [d]$, the function $m^{\HR}$ is supermodular
$$
m^{\HR}(ABC) + m^{\HR}(C) \ge m^{\HR}(AC) + m^{\HR}(BC),
$$
with equality ($=$ modular) if and only if the subsets satisfy the extremal conditional independence relation
\begin{equation*}
m^{\HR}(ABC) + m^{\HR}(C) = m^{\HR}(AC) + m^{\HR}(BC) \iff \Y_A \perp_e \Y_B \mid \Y_{C}.
\end{equation*}
\end{thm}

\begin{proof}
We assume without loss of generality that $A,B,C$ is a partition of the index set $[d]$. The conditional exponent measure density $\lambda(\y_{AB}\mid\y_C)=\frac{\lambda(\y)}{\lambda(\y_C)}$ of a H\"usler--Reiss vector $\Y$ with variogram $\Gamma$ and exponent measure density $\lambda(\y)$ is an $|AB|$-variate Gaussian with covariance matrix
\begin{align*}
	\Sigma^*&=-\frac{1}{2}\Gamma_{AB,AB} -\begin{pmatrix}
	-\frac{1}{2}\Gamma_{AB,C}& \1\\
	\end{pmatrix}
	\begin{pmatrix}
	-\frac{1}{2}\Gamma_{C,C} & \mathbf{1}\\
	\mathbf{1}^{\top}&0\\
	\end{pmatrix}^{-1}
	\begin{pmatrix}
	-\frac{1}{2}\Gamma_{C,AB}\\
	\1^{\top}\\
	\end{pmatrix},
\end{align*}
see Section~\ref{SUBSEC:EXTREMES}. Let $R^*$ be the correlation matrix of $\Sigma^*$ and note that $\Sigma^*=d_{\Sigma^*}R^*d_{\Sigma^*}$ where $d_{\Sigma^*}:=\sqrt{\diag(\Sigma^*)}$. 
As a result, we have 
\begin{align*}
\log\det(R^*)&=\log\det(\Sigma^*)-\sum_{i\in AB}\log(\Sigma^*_{ii})
\\
\log\det(R^*_{A,A})&=\log\det(\Sigma^*_{A,A})-\sum_{i\in A}\log(\Sigma^*_{ii}).
\end{align*}
Using standard properties of Schur complements, we furthermore find that
$$
\det(\Sigma^*) = \frac{\det\CM(\Gamma)}{\det\CM(\Gamma_{C,C})}
\quad\text{~and~}\quad
\det(\Sigma^*_{A,A}) = \frac{\det\CM(\Gamma_{AC,AC})}{\det\CM(\Gamma_{C,C})}.
$$
Together, we thus find that $-\frac{1}{2}\log\det (\Sigma^*)= m^{\HR}(ABC)-m^{\HR}(C)$ and similarly $-\frac{1}{2}\log\det (\Sigma^*_{A,A})= m^{\HR}(AC)-m^{\HR}(C)$.
By \Cref{THM:CIMULTI} and these identities, it then holds that 
\begin{align*}
    & -\frac{1}{2}\log(\det(R^*))+\frac{1}{2}\log(\det(R^*_{A,A}))+\frac{1}{2}\log(\det(R^*_{B,B}))\ge 0
    \\
    \Longleftrightarrow& -\frac{1}{2}\log(\det(\Sigma^*))+\frac{1}{2}\log(\det(\Sigma^*_{A,A}))+\frac{1}{2}\log(\det(\Sigma^*_{B,B}))\ge 0
    \\
    \Longleftrightarrow&~~ m^{\HR}(ABC)-m^{\HR}(C)- (m^{\HR}(AC)-m^{\HR}(C)) - (m^{\HR}(BC)-m^{\HR}(C))\ge 0
    \\
    \Longleftrightarrow&~~ m^{\HR}(ABC)+m^{\HR}(C) -m^{\HR}(AC)- m^{\HR}(BC)\ge 0,
\end{align*}
which confirms that $m^{\HR}$ is supermodular.

Finally, since the conditional exponent measure density is the density of an $|AB|$-variate Gaussian with correlation matrix $R^*$, we know by Theorem~\ref{THM:CIMULTI} that extremal conditional independence $\Y_A \perp_e \Y_B \mid \Y_{C}$ is equivalent to the equality
$$
-\frac{1}{2}\log\det(R^*)+\frac{1}{2}\log\det(R^*_{A,A})+\frac{1}{2}\log\det(R^*_{B,B})=0.
$$
Rewriting the lefthandside as above using $m^{\HR}$, this establishes the desired equivalence:
\begin{align*}
    m^{\HR}(ABC)+m^{\HR}(C) -m^{\HR}(AC)- m^{\HR}(BC)=0 \iff \Y_A\perp_e \Y_B \mid \Y_C.
\end{align*}
\end{proof}

\begin{cor}\label{COR:HRMI}
Under the same conditions as \Cref{THM:MULTIINFO}, we have $\Y_A \perp_e \Y_B \mid \Y_{C}$ if and only if
\begin{equation*}
    \det\CM(\Gamma_{ABC,ABC})\det\CM(\Gamma_{C,C}) \,-\,\det\CM(\Gamma_{AC,AC})\det\CM(\Gamma_{BC,BC}) \,=\, 0.
\end{equation*}
\end{cor}

To illustrate Theorem~\ref{THM:MULTIINFO} and Corollary~\ref{COR:HRMI} we return to our running example.
\begin{ex}
    Let $\Y$ be a H\"usler--Reiss graphical model with respect to the four-cycle as in Example~\ref{example: four-cycle}. This means that $\Y$ satisfies two extremal conditional independence statements $Y_1\perp_e Y_3\mid \Y_{24}$ and $Y_2\perp_e Y_4\mid \Y_{13}$. In terms of the result of \Cref{THM:MULTIINFO}, we get the following equalities:
    \begin{align*}
        m^{\HR}(1234)+m^{\HR}(24) -m^{\HR}(124)- m^{\HR}(234)&=0,\\
        m^{\HR}(1234)+m^{\HR}(13) -m^{\HR}(123)- m^{\HR}(134)&=0.
    \end{align*}
    Using Corollary~\ref{COR:HRMI}, under the assumption of a conditionally negative definite variogram we obtain
    \begin{align*}
        \det\CM(\Gamma_{1234,1234})\det\CM(\Gamma_{24,24})  
        &-
        \det\CM(\Gamma_{124,124})\det\CM(\Gamma_{234,234})=0,
        \\
        \det\CM(\Gamma_{1234,1234})\det\CM(\Gamma_{13,13})
        &-
        \det\CM(\Gamma_{123,123})\det\CM(\Gamma_{134,134})=0.
    \end{align*}
    These are polynomial equations in the entries of $\Gamma$, which explicitly expressed are
    \begin{align*}
        \left(
        \Gamma_{12}\Gamma_{23}
        -\Gamma_{14}\Gamma_{23}
        -\Gamma_{12}\Gamma_{24}
        +2\Gamma_{13}\Gamma_{24}
        -\Gamma_{14}\Gamma_{24}
        -\Gamma_{23}\Gamma_{24}
        +\Gamma_{24}^{2}
        -\Gamma_{12}\Gamma_{34}
        +\Gamma_{14}\Gamma_{34}
        -\Gamma_{24}\Gamma_{34}
        \right)^{2} &= 0,\\
        \left(
        \Gamma_{12}\Gamma_{13}
        -\Gamma_{13}^{2}
        -\Gamma_{12}\Gamma_{14}
        +\Gamma_{13}\Gamma_{14}
        +\Gamma_{13}\Gamma_{23}
        +\Gamma_{14}\Gamma_{23}
        -2\Gamma_{13}\Gamma_{24}
        +\Gamma_{12}\Gamma_{34}
        +\Gamma_{13}\Gamma_{34}
        -\Gamma_{23}\Gamma_{34}
        \right)^{2} &= 0.
    \end{align*}
    Since both $Y_1\perp_e Y_3\mid \Y_{24}$ and $Y_2\perp_e Y_4\mid \Y_{13}$ are conditional independence statements for singletons, by Lemma~\ref{LEM:APPENDIXCEGR} they are given by the vanishing minors
    \begin{align*}
        0&=\det \left( \CM(\Gamma_{124,234})\right)\\&\propto \Gamma_{12}\Gamma_{23}
        -\Gamma_{14}\Gamma_{23}
        -\Gamma_{12}\Gamma_{24}
        +2\Gamma_{13}\Gamma_{24}
        -\Gamma_{14}\Gamma_{24}
        -\Gamma_{23}\Gamma_{24}
        +\Gamma_{24}^{2}
        -\Gamma_{12}\Gamma_{34}
        +\Gamma_{14}\Gamma_{34}
        -\Gamma_{24}\Gamma_{34}, \\
        0&=\det \left( \CM(\Gamma_{123,134})\right)\\
        &\propto\Gamma_{12}\Gamma_{13}
        -\Gamma_{13}^{2}
        -\Gamma_{12}\Gamma_{14}
        +\Gamma_{13}\Gamma_{14}
        +\Gamma_{13}\Gamma_{23}
        +\Gamma_{14}\Gamma_{23}
        -2\Gamma_{13}\Gamma_{24}
        +\Gamma_{12}\Gamma_{34}
        +\Gamma_{13}\Gamma_{34}
        -\Gamma_{23}\Gamma_{34}.
    \end{align*}
    We observe that these are equivalent to the polynomial equations above.
\end{ex}

\section{Extremal conditional independence via \texorpdfstring{$\sigma^2$}{the resistance radius}}\label{SUBSEC:MODULARITYSIGMA}
In this section, we consider a second set function that encodes extremal conditional independence for H\"{u}sler--Reiss distributions. Recall the invariant $\sigma^2$ and vector $\p$, which naturally appears in the H\"{u}sler--Reiss exponent measure density $\lambda(\y)$; see \eqref{eq: Fiedler--Bapat identity} and \eqref{eq: HR density precision matrix}. By associating to each marginal $\Y_{(I)}$ the corresponding invariant $\sigma^2(I)$, we obtain a function on subsets of the index set. Our main result states that for a H\"{u}sler--Reiss distribution which satisfies EMTP$_2$ and $\p>0$, modularity of the set function $\sigma^2$ is equivalent to extremal conditional independence for the corresponding modular sets.

\subsection{Marginalization and properties of \texorpdfstring{$\p$ and $\sigma^2$}{the resistance curvature and resistance radius}}
From the Fiedler--Bapat identity \eqref{eq: Fiedler--Bapat identity}, one can immediately derive the following explicit expressions for $\p$ and $\sigma^2$ in terms of the variogram:
$$
\p = \frac{\Gamma^{-1}\1}{\1^T\Gamma^{-1}\1} \quad\text{~and~}\quad \sigma^2 = \frac{1}{2}\big(\1^T\Gamma^{-1}\1\big)^{-1}.
$$
Since marginalization of H\"{u}sler--Reiss distributions corresponds to taking submatrices at the level of the variogram, we likewise find that the same equations as above hold for $\p(I)$ and $\sigma^2(I)$, when replacing $\Gamma$ by $\Gamma_{I,I}$ for any nonempty $I\subseteq [d]$ with $\vert I\vert>1$; for single marginals $Y_i$, we define $\p(i)=1$ and $\sigma^2(i)=0$ for consistency with equation \eqref{eq: Fiedler--Bapat identity}. We furthermore note that $\sigma^2(ij)=\Gamma_{ij}/4$ for all $ij\subseteq [d]$ and thus that, as was the case for $m^{\HR}$, the set function $\sigma^2$ contains all information about $\Gamma$.

While the expressions above for $\p$ and $\sigma^2$ allow for immediate computation, we mention a result that is more useful when comparing between different margins $I',I\subseteq [d]$. The following results are an immediate consequence of the inverse relation \eqref{eq: Fiedler--Bapat identity} between the variogram and precision matrix.
\begin{prop}[{\citet[Theorem 3.2.3]{fiedler2011}}]\label{PROP:RESISTANCEKRON}
Let $\Theta$ be a H\"{u}sler--Reiss precision matrix. Then for nonempty $I\subseteq [d]$ we have
\begin{align*}
\Theta(I) &= \Theta_{I,I} - \Theta_{I,\backslash I}(\Theta_{\backslash I,\backslash I})^{-1}\Theta_{\backslash I,I}
\\
\p(I) &= \p_{I} - \Theta_{I,\backslash I}(\Theta_{\backslash I,\backslash I})^{-1}\p_{\backslash I}
\\
\sigma^2(I) &= \sigma^2 - \p_{\backslash I}^T(\Theta_{\backslash I,\backslash I})^{-1}\p_{\backslash I}.
\end{align*}
\end{prop}
\begin{proof}
Combining the Fiedler--Bapat identity for both $\Y$ and $\Y_{(I)}$, we find
\begin{align*}
    \begin{pmatrix}
\Theta(I) &\p(I)\\\p(I)^T&\sigma^2(I)
\end{pmatrix} = 
\begin{pmatrix}
-\Gamma_{I,I}/2 & \1\\\1^T&0
\end{pmatrix}^{-1}
&= 
\left(\begin{pmatrix}
-\Gamma/2&\1\\\1^T&0
\end{pmatrix}_{I+,I+}\right)^{-1}\\ &= \begin{pmatrix}
\Theta_{I,I}&\p_I\\\p_I^T&\sigma^2
\end{pmatrix} - \begin{pmatrix}
\Theta_{I,\backslash I}& \p_{\backslash I}
\end{pmatrix}(\Theta_{\backslash I,\backslash I})^{-1}\begin{pmatrix}
\Theta_{\backslash I,}\\ \p_{\backslash I}^T
\end{pmatrix}
,
\end{align*}
where $I+$ denotes the set $I$ including the last row/column index, and the last equality is a standard result for Schur complements. This completes the proof.
\end{proof}
The following two closure properties follow immediately.
\begin{prop}[{\citet{REZ2023,devriendt2022a}}]
If a H\"{u}sler--Reiss distribution is EMTP$_2$, then so are all of its marginals.
Furthermore, if a H\"{u}sler--Reiss distribution is EMTP$_2$ with $\p\geq0$ (respectively, $\p>0$), then so are all of its marginals.
\end{prop}
We list a number of equivalent representations of $\sigma^2(I)$: 
\begin{lemma}[\citet{HEER2026, devriendt2022a}]\label{LEM:sigmasqquotient}
Let $\Gamma$ be a variogram matrix and $I\subseteq [d]$ with $\vert I\vert \geq 1$. Then
\begin{enumerate}
    \item[(i)] $\sigma^2(I) = \frac{\det(-\Gamma_{I,I}/2)}{\det(\CM(\Gamma_{I,I}))}=\Big(\frac{-1}{2}\Big)^{\vert I\vert}\cdot \frac{\det(\Gamma_{I,I})}{\det(\CM(\Gamma_{I,I}))}.$
    \item[(ii)] $\sigma^2(I)=-2\log \int_{\y_I^\top\p(I)>0}\lambda_I(\y_I) d\y_I$,
    \item[(iii)] $\sigma^2(I) = \max\{\tfrac{1}{2}\x^T\Gamma_{I,I}\x \,:\, \1^T\x=1\}$,
    \item[(iv)] $\sigma^2(I) = -\frac{1}{4}\sum_{i,j\in I}\Theta(I)_{ij}(\Gamma_{ik}-\Gamma_{jk})^2$ for any $k\in I$,
    \item[(v)] $\sigma^2(I) = \frac{1}{4\vert I\vert}\operatorname{Tr}(\Gamma_{I,I}\Theta(I)\Gamma_{I,I}).$
\end{enumerate}
\end{lemma} 
\begin{proof}
The first expression follows by the Fiedler--Bapat identity and the cofactor formula for matrix inversion. The second expression is shown in \citet{HEER2026} and expressions (iii)-(v) are shown in \citet[Chapter 6]{devriendt2022a}.
\end{proof}

\subsection{Extremal conditional independence}
Under the conditions of \EMTPtwo and $\p>0$, the following theorem shows that modularity of $\sigma^2$ encodes extremal conditional independence for H\"usler--Reiss distributions, similar to $m^{\HR}$. We note that submodularity and the fact that modularity implies graph separation statements were proven before in \cite[Chapter 6]{devriendt2022a} in the context of graph theory and Laplacian and effective resistance matrices.
\begin{thm}\label{THM:SIGMASQ}
Let $\Y$ be a H\"{u}sler--Reiss vector that satisfies EMTP$_2$ and $\mathbf{p}\geq 0$. Then for any nonempty disjoint $A,B,C\subset [d]$, the function $\sigma^2$ is submodular
$$
\sigma^2(ABC) + \sigma^2(C) \leq\sigma^2(AC)+\sigma^2(BC),
$$
with equality ($=$ modular) if and only if the subsets satisfy the extremal conditional independence relation:
$$
\sigma^2(ABC) + \sigma^2(C) = \sigma^2(AC)+\sigma^2(BC) \iff \Y_A\perp_e\Y_B \mid \Y_C.
$$
\end{thm}
\begin{proof}
    Without loss of generality, we may assume that $A,B,C$ is a partition of $[d]$. Under the assumption that $C$ is nonempty and that $\Y$ is \EMTPtwok we have that $L := \Theta_{AB,AB}$ is an M-matrix, since it is a positive definite matrix with non-positive off-diagonal entries, where positive definiteness holds because it is a submatrix of the Laplacian matrix of a connected graph \citep[Property 3.26]{devriendt2022a}.
    Since $K := L^{-1}$ is the inverse of an M-matrix it has nonnegative entries \citep[Theorem A.3.2]{fiedler2011}. Using Schur complements, we obtain
    \begin{align*}
        (L_{A,A})^{-1} &= K_{A,A} - K_{A,B}K_{B,B}^{-1}K_{A,B}^\top,\\
        (L_{B,B})^{-1} &= K_{B,B} - K_{A,B}^\top K_{A,A}^{-1}K_{A,B}.
    \end{align*}
    Using Proposition \ref{PROP:RESISTANCEKRON}, we can thus write:
    \begin{align*}
        &\sigma^2(ABC) =: \sigma^2,\\
        &\sigma^2(AC) = \sigma^2 - \p_B^{\top} (L_{B,B})^{-1} \p_B = \sigma^2 - \p_B^{\top} (K_{B,B} - K_{A,B}^\top K_{A,A}^{-1}K_{A,B}) \p_B,\\
        &\sigma^2(BC) = \sigma^2 - \p_A^{\top} (L_{A,A})^{-1} \p_A = \sigma^2 - \p_A^{\top} (K_{A,A} - K_{A,B}K_{B,B}^{-1}K_{A,B}^\top) \p_A,\\
        &\sigma^2(C) = \sigma^2 - \left(\p_A^\top K_{A,A} \p_A + \p_B^\top K_{B,B} \p_B + 2\p_A^\top K_{A,B}\p_B \right).
    \end{align*}
    Combining everything, we have
    \begin{align}
        \sigma^2(AC)+\sigma^2(BC)&-\sigma^2(ABC) - \sigma^2(C) \label{EQ:SIGMA}\\
        &=\p_B^{\top}K_{A,B}^\top K_{A,A}^{-1}K_{A,B}\p_B +  \p_A^{\top}K_{A,B}K_{B,B}^{-1}K_{A,B}^\top\p_A +2\p_A^\top K_{A,B}\p_B. \nonumber
    \end{align}
Due to positive definiteness of $K_{A,A}$ and $K_{B,B}$, the first two summands above are nonnegative. Furthermore, since $\p$ is nonnegative and $K_{A,B}$ is nonnegative as a submatrix of an inverse $M$-matrix, we furthermore know that $\p_A^{\top}K_{A,B}\p_B \geq 0$. This proves the submodularity statement in the theorem.

If $\p>0$, the sum in equation \eqref{EQ:SIGMA} is equal to zero if and only if $K_{AB} = \0$, since otherwise there would be at least one positive summand. By the cofactor expression for matrix inversion, the entries of $K_{A,B}=(\Theta_{AB,AB})^{-1}_{A,B}$ are given by the minors $\det(\Theta_{\backslash Ci, \backslash Cj})$, where $i \in A, j \in B$. 
    The vanishing of these minors is equivalent to
    \begin{equation*}
        Y_i \perp_e Y_j \mid \Y_C,
    \end{equation*}
    see Lemma~\ref{LEM:APPENDIXCEGR}.
    By \citet[Remark~4.3]{REZ2023}, we have under \EMTPtwo that
    $\Y_A \perp_e \Y_B \mid \Y_C$
    is equivalent to the collection of all statements $Y_i \perp_e Y_j \mid \Y_C$ for all pairs $i\in A$, $j\in B$.
    This concludes the proof.
\end{proof}

\begin{ex}
    Let $\Y$ be a H\"usler--Reiss graphical model with respect to the four-cycle as in Example~\ref{example: four-cycle}, and which satisfies EMTP$_2$ and $\p>0$. By the pairwise Markov property, $\Y$ satisfies two extremal conditional independence statements $Y_1\perp_e Y_3\mid \Y_{24}$ and $Y_2\perp_e Y_4\mid \Y_{13}$. By Theorem \ref{THM:SIGMASQ}, these extremal conditional independences imply the following modularity equations:
    \begin{align*}
        \sigma^2(1234) + \sigma^2(24) &= \sigma^2(124) + \sigma^2(234), \\
        \sigma^2(1234) + \sigma^2(13) &= \sigma^2(123) + \sigma^2(134).
    \end{align*}
    These modularity equations can be reduced to the following polynomial equations:
    \begin{align*}
    q_{13\mid 24}(\Theta,\p)
    &:= \frac{\Theta_{13}}{\Theta_{11}\Theta_{33}\big(\Theta_{11}\Theta_{33}-\Theta_{13}^{2}\big)}\Big(p_{1}^{2}\,\Theta_{13}\Theta_{33}
    - 2\,p_{1}p_{3}\,\Theta_{11}\Theta_{33}
    + p_{3}^{2}\,\Theta_{11}\Theta_{13}\Big) \;=\; 0,\\[0.3em]
    q_{24\mid 13}(\Theta,\p)
    &:= \frac{\Theta_{24}}{\Theta_{22}\Theta_{44}\big(\Theta_{22}\Theta_{44}-\Theta_{24}^{2}\big)}\Big(p_{2}^{2}\,\Theta_{24}\Theta_{44}
    - 2\,p_{2}p_{4}\,\Theta_{22}\Theta_{44}
    + p_{4}^{2}\,\Theta_{22}\Theta_{24}\Big) \;=\; 0.
    \end{align*}
    We observe that the denominators are nonzero, as they are determinants of positive definite submatrices. Furthermore, the equations hold under extremal conditional independence by Lemma~\ref{LEM:APPENDIXCEGR}, as $Y_1\perp_e Y_3\mid \Y_{24}\Leftrightarrow \Theta_{13}=0$ and $Y_2\perp_e Y_4\mid \Y_{13}\Leftrightarrow \Theta_{24}=0$.
    However, we can only certify equivalence under the conditions of Theorem~\ref{THM:SIGMASQ}, because only then the polynomials inside the brackets are guaranteed to be nonzero.
\end{ex}
To conclude, we note that even though the full distribution may not satisfy EMTP$_2$ and $\p>0$, Theorem \ref{THM:SIGMASQ} still applies when restricted to a marginal $\Y_{(I)}$ which does satisfy these properties.
\begin{cor}\label{cor: CI from sigma for marginals}
Let $\Y$ be a H\"{u}sler--Reiss vector. Then for any nonempty disjoint subsets $A,B,C\subset [d]$ such that the marginal vector $\Y_{(ABC)}$ is EMTP$_2$ with $\p>0$, we have 
$$
\Y_A\perp_e \Y_B \mid \Y_C \iff \sigma^2(ABC)+\sigma^2(C)=\sigma^2(AB)+\sigma^2(BC).
$$
\end{cor}
\begin{proof}
This follows because conditional independence for $A,B,C$ relative to $\Y$ is equivalent to conditional independence for $A,B,C$ relative to $\Y_{(ABC)}$. The corollary then follows by Theorem \ref{THM:SIGMASQ}.
\end{proof}

\section{Geometry of variogram matrices}\label{sec:geometry}

The set of all valid correlation matrices of a multivariate Gaussian distribution is a bounded subset of the positive semidefinite cone. This is typically referred to as the Gaussian elliptope.
\begin{defi}[Gaussian elliptope]
    The \emph{$d$-dimensional Gaussian elliptope} is defined as:
    \begin{equation*}
        \mathcal{E}_d := \left\{ R \in \mathbb{R}^{d \times d} \,\middle|\, 
        R = R^\top,\; R \text{ is positive semidefinite},\; \mathrm{diag}(R) = \mathbf{1} \right\}.
    \end{equation*}
We also write $\mathcal{E}_d\subset\mathbb{R}^{\binom{d}{2}}$ for the projection of this set onto the coordinates corresponding to the upper-triangular off-diagonal entries.
\end{defi}
Since the matrices in the elliptope are symmetric and have diagonal elements fixed to one, the dimension of this set without further constraints is $\binom{d}{2}$. It is a bounded set, since by construction the correlations of two Gaussians is bounded in the interval $[-1,1]$. The boundary $\partial\mathcal{E}_d$ of the Gaussian elliptope is the intersection of $\mathcal{E}_d$ with the hypersurface of singular matrices, cut out by the polynomial $\det(R)=0$. Its interior $\mathcal{E}_d^\circ$ contains positive definite correlation matrices, which is the generic case.

We aim to find and study a similar object for H\"usler--Reiss distributions. In the case of Gaussians, the standard parameterization via (positive semidefinite) covariance matrices is normalized by rescaling rows and columns to obtain a diagonal of ones. For H\"{u}sler--Reiss distributions, the standard parameterization is via (conditionally negative definite) variograms $\mathcal{D}^d$, but this is again an unbounded set. In order to obtain a bounded set, we consider variograms with bounded $\sigma^2$. Taking the upper-triangular off-diagonal entries, we can again think of this as a subset of $\mathbb{R}^{\binom{d}{2}}$.
\begin{defi}[H\"{u}sler--Reiss elliptope]
    We define the \emph{$d$-dimensional Hüsler--Reiss elliptope} as:
        \begin{equation*}
        \mathcal{F}_d := \left\{ \Gamma \in \mathcal{D}^{d} \,\middle|\, \sigma^2(\Gamma) \leq 1 \right\}.
    \end{equation*}
\end{defi}
\begin{remark}
Any variogram $\Gamma\in\mathcal{D}^d$ corresponds to a Euclidean distance matrix that contains the squared Euclidean distances between $d$ affinely independent points in $\RR^{d-1}$. In other words, the entries $\Gamma_{ij}$ are the squared edge lengths of a simplex. Moreover, the squared circumradius of this simplex is $\sigma^2$ and the barycentric coordinate of the circumcenter is $\p$; see \cite{fiedler2011}. In particular, this means that the H\"{u}sler--Reiss elliptope corresponds to the set of all Euclidean distance matrices of simplices with circumradius at most one.
\end{remark}

We can relate the H\"{u}sler--Reiss elliptope to the Gaussian elliptope by the following map:
\begin{prop}\label{prop: R and Gamma}
If $\Gamma$ is a variogram, then $R(\Gamma):=-\tfrac{1}{2\sigma^2}\Gamma+\1\1^T$ is a correlation matrix of rank $d-1$ with $\ker(R) = \spa(\mathbf{p})$ and thus $\mathbf{p}^T\1 \neq  0$. If $R$ is a correlation matrix of rank $d-1$ with $\ker(R) = \spa(\mathbf{w})$ and $\mathbf{w}^T\1 \neq 0$, then $\Gamma(R) := 2(\1\1^T-R)$ is a variogram with $\sigma^2 = 1$ and $\p=\mathbf{w}/(\mathbf{w}^T\1)$.
\end{prop}
\begin{proof}
For any nonzero vector $\x\in\spa(\1)^\perp$, the quadratic form $\x^T R(\Gamma)\x=-\frac{1}{2\sigma^2}\x^T\Gamma\x$ is positive by conditional negative definiteness, so $R(\Gamma)$ has at least $\dim({\spa(\1)}^\perp)=d-1$ positive eigenvalues. Since $\p^T\1=1$ by definition of the vector $\p$, we find $R(\Gamma)\p=0$ and thus, by the former part that $\ker(R(\Gamma))=\spa(\p)$ and that it is a positive semidefinite matrix. This proves that $R(\Gamma)$ is a correlation matrix.

For the converse, we proceed again by considering the quadratic form $\x^T\Gamma(R)\x=-2\x^TR\x$ with nonzero $\x\in\spa(\1)^\perp$. By positive semidefiniteness of $R$ and $\mathbf{w}^T\1\neq 0$, it follows that this quadratic form is negative and thus that $\Gamma$ is a variogram. Since $\Gamma(R)\mathbf{w}=2(\1\1^T-R)\mathbf{w}=2(\1^T\mathbf{w}).\mathbf{w}$, the claim about $\p$ and $\sigma^2$ follow from their defining equation $\Gamma(R)\p=2\sigma^2\p$ and $\p^T\1=1$.
\end{proof}
\begin{prop}\label{PROP:ELLIPTOPE}
Consider the following subset of the boundary of the Gaussian elliptope:
$$
\partial \mathcal{E}_d^* := \big\{R \in \mathcal{E}_d \mid \ker(R)=\spa(\mathbf{w}) \textup{~with~}\mathbf{w}^T\1\neq 0\big\} \subset\partial\mathcal{E}.
$$
Then the H\"{u}sler--Reiss elliptope, as a subset of $\mathbb{R}^{\binom{d}{2}}$, is equal to
$$
\mathcal{F}_d \,=\, 2(\1\1^T-\mathcal{E}_d^\circ\cup\partial\mathcal{E}_d^*) \,=\, \bigcup\limits_{t\in(0,1]}2t(\1\1^T-\partial\mathcal{E}_d^*).
$$
\end{prop}
\begin{proof}
Recall that the boundary of the elliptope contains singular correlation matrices so indeed $\partial\mathcal{E}^*_d\subset\partial\mathcal{E}$. We first show the second equality (the union expression) and then the first.

We will show that the map $\varphi:\widetilde{R} \mapsto 2(\1\1^T-\widetilde{R})$ is a bijection between $\mathcal{S}:=\bigcup_{t\in(0,1]}2t\cdot(\1\1^T-\partial\mathcal{E}^*_d)$ and $\mathcal{F}_d$, both taken as subsets of $\mathbb{R}^{d\times d}$. The claimed equality for $\mathcal{F}_d$ then follows by projecting onto the $\binom{d}{2}$ upper-diagonal entries. \emph{(Surjectivity)} For any $\Gamma\in \mathcal{F}_d$, we know by Proposition \ref{prop: R and Gamma} that the matrix $R(\Gamma)=-\frac{1}{2\sigma^2}\Gamma +\1\1^T$ is a correlation matrix of rank $d-1$ with $\ker(R)=\spa(\mathbf{w})$ and $\mathbf{w}^T\1\neq 0$, and thus that $R(\Gamma)\in \mathcal{S}$. Since $0<\sigma^2\leq 1$, we furthermore know that $\sigma^2  R(\Gamma)\in\mathcal{S}$ and that $\varphi(\sigma^2R(\Gamma))=\Gamma$; as $\Gamma$ was arbitrary, we have shown that $\varphi$ is surjective. 
\emph{(Injectivity)} Let $\widetilde{R},\widehat{R}\in\mathcal{S}$ be such that $\varphi(\widetilde{R})=\varphi(\widehat{R})=\Gamma$, then by definition of the map $\varphi$, it immediately follows that $\widetilde{R}=\widehat{R}$. This proves injectivity and thus establishes the second expression for $\mathcal{F}_d$.

The proof of the first statement follows similarly. For the boundary part $\partial\mathcal{E}_d^*$, the equality follows from the second statement with $t=1$. For the interior, we consider the map $\psi:\widetilde{R}\mapsto2(\1\1^T-\widetilde{R})$ from $\mathcal{E}_d^\circ$ to $\mathcal{F}_d$, both taken as subsets of $\mathbb{R}^{d\times d}$. The image of $\widetilde{R}\in\mathcal{E}^\circ_d$ under this map is a variogram since $\x^T(\1\1^T-\widetilde{R})\x<0$ for all nonzero $\x\perp \1$ by positive definiteness of $\widetilde{R}$. \textit{(Injectivity)} Since $\psi(\widetilde{R})=\psi(\widehat{R})\Leftrightarrow2(\1\1^T-\widetilde{R})=2(\1\1^T-\widehat{R})\Rightarrow \widetilde{R}=\widehat{R}$ the map is injective. \textit{(Surjectivity)} Let $\Gamma$ be any variogram with $\sigma^2<1$. Then we know that $\psi(-\frac{1}{2}\Gamma+\1\1^T)=\Gamma$ by construction. We show that $-\frac{1}{2}\Gamma+\1\1^T$ lies in $\mathcal{E}^\circ$: it is symmetric, has unit diagonal and it is positive definite because $\x^T(-\tfrac{1}{2}\Gamma+\1\1^T)\x >0$ for all nonzero $\x\perp\1$ and
$$
\1^T(-\tfrac{1}{2}\Gamma+\1\1^T)\1 = d^2- \tfrac{1}{2}\1^T\Gamma\1 \geq d^2-d^2\sigma^2> 0,
$$
where the first inequality follows because $\sigma^2=\tfrac{1}{2}\max \{\x^T\Gamma\x:\x^T\1=1\}$ (see expression (iii) in Lemma \ref{LEM:sigmasqquotient}) and the second one because $\sigma^2<1$. This completes the proof after projecting onto $\mathbb{R}^{\binom{d}{2}}$.
\end{proof}
\begin{remark}
The second expression for $\mathcal{F}_d$ in Proposition \ref{PROP:ELLIPTOPE} suggests that the subset $\partial\mathcal{E}^*_d$ of the Gaussian elliptope boundary already describes the H\"{u}sler--Reiss elliptope up to rescaling by $t\in(0,1]$. This suggest that one could also focus on variograms with $\sigma^2=1$.
\end{remark}
\begin{ex}[$d=3$]
In the $3$-dimensional case, we can visualize the H\"{u}sler--Reiss elliptope $\mathcal{F}_3\subset\mathbb{R}^3$, with coordinates corresponding to the off-diagonal variogram entries $\Gamma_{12},\Gamma_{13},\Gamma_{23}$. 
As prescribed by Proposition \ref{PROP:ELLIPTOPE}, we either consider the Gaussian elliptope with some points on the boundary removed (first expression) or (i) start from the boundary of the Gaussian elliptope $\mathcal{E}_3$ with correlation matrices of rank $<2$ or with kernel spanned by $\mathbf{w}\perp \1$ removed, (ii) apply the transformation $2(\1\1^T-R)$ from Proposition \ref{prop: R and Gamma} and (iii) take all $t$-scaled points, for $t\in (0,1]$. Figure \ref{fig:extremalelliptope} shows the resulting set, with points in red indicating the points removed from the boundary.
\begin{figure}[h!]
    \centering
    \includegraphics[width=0.8\linewidth]{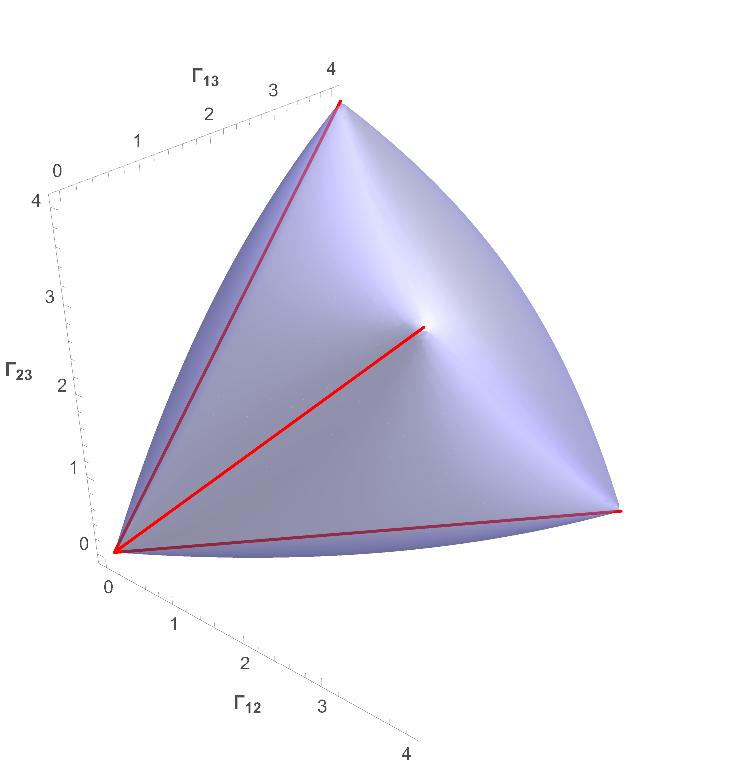}
    \caption{The $3$-dimensional H\"usler--Reiss elliptope $\mathcal{F}_3\subset\mathbb{R}^3$. The red lines correspond to correlation matrices removed as described in Proposition \ref{PROP:ELLIPTOPE}.}
    \label{fig:extremalelliptope}
\end{figure}

Next, we consider the intersection of $\mathcal{F}_d$ with the additional constraint that $\p\geq 0$. For $d=3$, this is equivalent to the EMTP$_2$ condition. Figure \ref{fig:elliptopelinear} shows the resulting set. We furthermore highlight that the linear faces of the resulting set are precisely the three possible conditional independence constraints in the $3$-dimensional case, i.e., the $3$ possible connected graphical models on $3$ variables:
$$
Y_i\perp_e Y_j\mid Y_k \iff \Gamma_{ij}=\Gamma_{ik}+\Gamma_{jk}
$$
for $i,j,k$ distinct, corresponding to a graphical model for the path $i-k-j$. The plane in red corresponds to the constraint $Y_1\perp_e Y_2\mid Y_3$ and the plane in green corresponds to the constraint $Y_2\perp_e Y_3\mid Y_1$. The plane corresponding to the constraint $Y_1\perp_e Y_3\mid Y_2$ lies in the opposite side of the elliptope.
\end{ex}

\begin{figure}
    \centering
    \includegraphics[width=0.8\linewidth]{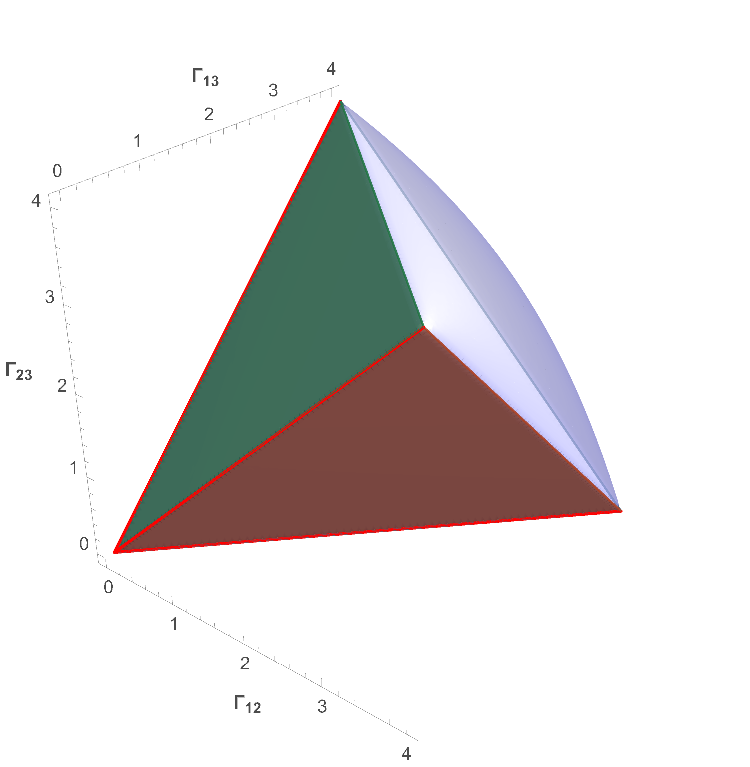}
    \caption{The H\"usler--Reiss elliptope $\mathcal{F}_3$ intersected with the region where $\p \geq 0$ holds or, equivalently for $d=3$, where \EMTPtwo holds.}
    \label{fig:elliptopelinear}
\end{figure}

\section*{Acknowledgements}
We thank Tobias Boege for the suggestion to study multiinformation for H\"usler--Reiss distributions. For the purpose of open access, the authors have applied a CC BY public copyright licence to any author accepted manuscript arising from this submission.

    \bibliographystyle{chicago}
	\bibliography{bibliography}

\end{document}